\documentclass[12pt,reqno]{amsart}
\usepackage{amssymb,amsmath,amsthm,hyperref}
\usepackage{xcolor}
\newcommand{\sg}{\textnormal{sg}}

\newtheorem{theorem}{Theorem}
\newtheorem{lemma}[theorem]{Lemma}
\newtheorem{corollary}[theorem]{Corollary}
\newtheorem{proposition}[theorem]{Proposition}

\theoremstyle{remark}

\theoremstyle{definition}

\numberwithin{theorem}{section} 
\numberwithin{equation}{section}
\numberwithin{example}{section}
\newcommand{\ch}{{\text {\rm ch}}}
\newcommand{\mult}{{\text {\rm mult}}}

\newcommand{\Z}{\mathbb{Z}}

\renewcommand{\Im}{\text {\rm Im}}

\title{On string functions and double-sum formulas}

\begin{document}

\date{12 July 2021}



\begin{abstract}
String functions are important building blocks of characters of integrable highest modules over affine Kac--Moody algebras. Kac and Peterson computed  string functions for affine Lie algebras of type $A_{1}^{(1)}$ in terms of Dedekind eta functions.  We produce new relations between string functions by writing them as double-sums and then using certain symmetry relations. We evaluate the series using special double-sum formulas that express Hecke-type double-sums in terms of Appell--Lerch functions and theta functions, where we point out that Appell--Lerch functions are the building blocks of Ramanujan's classical mock theta functions. 
\end{abstract}

\author{Eric T. Mortenson}
\address{Department of Mathematics and Computer Science, Saint Petersburg State University, Saint Petersburg, 199034, Russia}
\email{etmortenson@gmail.com}

\author{Olga Postnova}
\address{Euler International Mathematical Institute, Laboratory of Mathematical Problems of Physics, Saint Petersburg Department of Steklov Institute of Mathematics, Saint Petersburg, Fontanka river emb. 27,
191023 Saint Petersburg,
Russia}
\email{postnova.olga@gmail.com}

\author{Dmitry Solovyev}
\address{Department of Physics, Saint Petersburg State University,  Ulyanovkaya str.1, Saint Peterburg, Russia; Department of Mathematics, Indiana University, Bloomington, IN 47405, USA}
\email{dimsol42@gmail.com}
\maketitle
\setcounter{section}{-1}
\setcounter{tocdepth}{1}  


\section{Notation}\label{section:notation}
 Let $q$ be a complex number where $q:=e^{2\pi i \tau}$ and $\tau\in\mathfrak{H}:=\{z\in \mathbb{C} | \Im{z}>0\}$.  Define $\mathbb{C}^*:=\mathbb{C}-\{0\}$.  We recall some basic notation:
\begin{gather*}
(x)_n=(x;q)_n:=\prod_{i=0}^{n-1}(1-q^ix), \ \ (x)_{\infty}=(x;q)_{\infty}:=\prod_{i= 0}^{\infty}(1-q^ix),\\
{\text{and }}\ \ j(x;q):=\sum_{n=-\infty}^{\infty}(-1)^nq^{\binom{n}{2}}x^n=(x)_{\infty}(q/x)_{\infty}(q)_{\infty},
\end{gather*}
where in the last line the equivalence of product and sum follows from Jacobi's triple product identity.  We note that $j(q^n;q)=0$ for $n\in \mathbb{Z}.$  Let $a$ and $m$ be integers with $m$ positive.  Define
\begin{gather*}
J_{a,m}:=j(q^a;q^m), \ \ \overline{J}_{a,m}:=j(-q^a;q^m), \ J_m:=J_{m,3m}=\prod_{i= 1}^{\infty}(1-q^{mi}),
\end{gather*}
and
\begin{equation*}
\eta(\tau):=q^{\tfrac{1}{24}}\prod_{n=1}^{\infty}(1-q^{n}).
\end{equation*}
We will use the following definition of an Appell--Lerch function.   Let $x,z\in\mathbb{C}^*$ with neither $z$ nor $xz$ an integral power of $q$. Then
\begin{equation}
m(x,q,z):=\frac{1}{j(z;q)}\sum_r\frac{(-1)^rq^{\binom{r}{2}}z^r}{1-q^{r-1}xz}.\label{equation:mxqz-def}
\end{equation}
We define our Hecke-type double-sum as follows.  Let $x,y\in\mathbb{C}^*$, then
\begin{equation}
f_{a,b,c}(x,y,q):=\Big ( \sum_{r,s\ge 0 }-\sum_{r,s<0}\Big )(-1)^{r+s}x^ry^sq^{a\binom{r}{2}+brs+c\binom{s}{2}},\label{equation:fabc-def}
\end{equation}
which has the alternate form
\begin{equation}
f_{a,b,c}(x,y,q)
=\sum_{\substack{r,s\in\mathbb{Z}\\ \sg(r)=\sg(s)}}\sg(r)(-1)^{r+s}x^ry^sq^{a\binom{r}{2}+brs+c\binom{s}{2}},
\end{equation}
where
\begin{equation}
\sg(r):=\begin{cases}
1 & \textup{if} \ r\ge0,\\
-1 & \textup{if} \  r<0.
\end{cases}
\end{equation}

\section{Introduction}\label{section:introduction}
 The goal of this paper is to find new formulae for string functions for affine Kac--Moody algebras. Let us start with brief reminder of what a string function is in this case. For details, see \cite{Kac, KP}.

\smallskip
For the affine Kac--Moody algebra $\mathfrak{g}=A_1^{(1)}$ the simple roots are $\alpha_0=\delta-\alpha$ and $\alpha_1=\alpha$, where $\alpha$ is the root of $A_1$ Lie subalgebra and $\delta$ is the imaginary root. Denote by $\Lambda_0$ and $\Lambda_1$ corresponding fundamental weights. 
Denote by $P_+$ the subset of integral dominant weights in the weight lattice $\mathbb{Z}\Lambda_0\bigoplus\mathbb{Z}\Lambda_1,\;\; P_+=\{a\Lambda_0+b\Lambda_1, a,b\in\mathbb{Z}_{\geq 0}\}$.

\smallskip
Let $L(\Lambda), \Lambda\in P_+$  be  an irreducible $\mathfrak{g}$-module with highest weight $\Lambda\in P_+$ of level $N$:
\begin{equation*}
\Lambda=(N-\ell)\Lambda_0+\ell\Lambda_1.
\end{equation*}
Define the character  of $L(\Lambda)$ as a function
\begin{equation*}
\ch_{L(\Lambda)}(z,q)=\sum_{\lambda\in P(\Lambda)\subset \mathfrak{h}^{*}}\mult_{\Lambda}(\lambda)q^Az^{A-B},
\end{equation*}
Here a sum is taken over all weights of a module, these are weights  that occur in $L(\lambda)$, i.e
$\lambda=\Lambda-A\alpha_0-B\alpha_1$, where $A,B\in \mathbb{Z}_{\geq 0}$.   

\smallskip
It is possible to express the character via the subset of weights of a module called maximal weights. A weight $\lambda\in P(\Lambda)$ is called \textit{maximal} if $\lambda+\delta\not\in P(\Lambda)$. We will parametrize maximal weights of the module as
\begin{equation*}
\lambda=(N-m)\Lambda_0+m\Lambda_1.
\end{equation*}
For a maximal weight $\lambda \in P(\Lambda)$ the \textit{string function} of $\lambda$  is defined as
\begin{equation}
\label{strfun}
c^\Lambda_{\lambda}(q)=c^{N-\ell,\ell}_{N-m,m}(q):=q^{ s_{\Lambda}(\lambda)}\sum_{n\geq 0} \mult_{\Lambda}(\lambda-n\delta)q^{n}=: C_{m,\ell}^N(q),
\end{equation}
where
\begin{equation}
s_{\Lambda}(\lambda):=s(m,\ell,N)=-\frac{1}{8}+\frac{(\ell+1)^2}{4(N+2)}-\frac{m^2}{4}.\label{equation:s-def}
\end{equation}

\smallskip
The character of $L(\Lambda)$ can be expressed in terms of theta functions and string functions as (See (\cite{SW,Kac,KP} for details and for how to extend it to other Kac--Moody algebras)
\begin{equation*}
\ch_{L(\Lambda)}(z,q)=\sum_{\substack{0\leq m <2N\\ \text{$m+\ell$ even}}}
C_{m,\ell}^{N}(q)\Theta_{m,N}(z,q), 
\end{equation*}
where $\Theta_{n,m}(z,q)$ is a Jacobi theta function of degree $m$ and characteristic $n$
\begin{equation*}\label{Thetadef}
\Theta_{n,m}(z,q)=\sum_{j\in\Z+n/2m} q^{mj^2}z^{-mj}.
\end{equation*}
The  $C_{m,\ell}^N(q)$ are level-$N$ A$_1^{(1)}$ string functions, see (\ref{strfun}).

\smallskip
The character can be computed as (see \cite{SW, W} and references therein)

\begin{align*}
\ch_{L(\Lambda)}(z,q)=\frac{1}{\eta^3(\tau)}
\sum_{m\in 2\Z+\ell} \sum_{j\in\Z}\sum_{i\in\mathbb{N}}
(-1)^i q^{\frac{1}{2}i(i+m)+ (N+2)(j+(\ell+1)/(2(N+2)))^2} \\
\times \Bigl\{q^{\frac{1}{2}i(2(N+2)j+\ell+1)}
-q^{-\frac{1}{2}i(2(N+2)j+\ell+1)}\Bigr\}z^{-\frac{1}{2}m},
\end{align*}
which gives the expression for string functions as
\begin{multline}\label{sffinal}
C_{m,\ell}^N(q)=\frac{q^{\frac{(\ell+1)^2}{4(N+2)}-\frac{m^2}{4N}}}
{\eta^3(\tau)}\sum_{j\in\Z}
\sum_{i\in\mathbb{N}} (-1)^i q^{\frac{1}{2}i(i+m)+j((N+2)j+\ell+1)} \\
\times
\Bigl\{q^{\frac{1}{2}i(2(N+2)j+\ell+1)}-q^{-\frac{1}{2}i(2(N+2)j+\ell+1)}\Bigr\}.
\end{multline}
This is our starting point.

\smallskip
Before we state our main results, we recall the classic symmetries for string functions.  They are
{\allowdisplaybreaks \begin{align*}
C_{m,\ell}^{N}(q)&=C_{-m,\ell}^{N}(q),\\
C_{m,\ell}^{N}(q)&=C_{2N-m,\ell}^{N}(q),\\
C_{m,\ell}^{N}(q)&=C_{N-m,N-\ell}^{N}(q).
\end{align*}}%
We also remind the reader of our definition for $s_{\Lambda}(\lambda)$ (\ref{equation:s-def}).  Our main results read
\begin{theorem} \label{theorem:main-result} We have
{\allowdisplaybreaks \begin{align}
C_{m,\ell}^{2K}(q)\pm C_{2K-m,\ell}^{2K}(q)
&=\frac{q^{s(m,\ell,2K)}}{J_{1}^3}\Big ( f_{K+1,K+1,1}(\pm q^{1+\frac{1}{2}(K+\ell)},q^{1+\frac{1}{2}(m+\ell)},q)
\label{equation:OP-split-1}\\
& \ \ \ \ \  \pm q^{\frac{1}{2}(K-\ell)}f_{K+1,K+1,1}(\pm q^{1+\frac{1}{2}(3K-\ell)},q^{1+K+\frac{1}{2}(m-\ell)},q)\Big).\notag
\end{align}}%
\end{theorem}
\begin{corollary} \label{corollary:f1K1-fKK1-A} We have
{\allowdisplaybreaks \begin{align}
C_{m,K}^{2K}(q)
&=\frac{q^{s(m,K,2K)}}{J_{1}^3}f_{K+1,K+1,1}(q^{K+1},q^{1+\frac{1}{2}(m+K)},q).
\label{equation:OP-2}
\end{align}}%
\end{corollary}

\begin{corollary} \label{corollary:f1K1-fKK1-B} For $K\equiv \ell \pmod 2$, we have
{\allowdisplaybreaks \begin{align}
C_{K,\ell}^{2K}(q)&=\frac{q^{s(K,\ell,2K)}}{J_{1}^3}
f_{K+1,K+1,1}(q^{1+\frac{1}{2}(K+\ell)},q^{1-\frac{1}{2}(K-\ell)},q).
\label{equation:OP-3}
\end{align}}%
\end{corollary}

Theorem \ref{theorem:main-result} and its two corollaries provide a framework for interpreting and proving the string function expressions for affine Lie algebras of type $A_{1}^{(1)}$ found in Kac and Peterson \cite{KP}.  In \cite[pp. 219-220]{KP}, Kac and Peterson give several examples of string functions for affine Lie algebras of type $A_{1}^{(1)}$ that have beautiful evaluations in terms of theta functions.  See also \cite{KW1, KW2}.  If we fix a positive integer $m$, their string functions are of the form \cite[p. 260]{KP}:
\begin{equation}
c_{\lambda}^{\Lambda}(\tau)
=\frac{1}{\eta(\tau)^3}\cdot\sum_{\substack{(x,y)\in\mathbb{R}^2\\ -|x|<y\le |x|\\
(x,y) \ \textup{or} \ (1/2-x,1/2+y)\in ((N+1)/2(m+2),n/2m)+\mathbb{Z}^2}}
\sg(x)q^{(m+2)x^2-my^2},\label{equation:KP-fullGlory}
\end{equation}
where $N$ and $n$ are integers with $n\equiv N \pmod 2$.  The string functions $c_{\lambda}^{\Lambda}(\tau)$ are closely related to the real quadratic fields $\mathbb{Q}(\sqrt{m(m+2)})$ and to Hecke indefinite modular forms.  
We have replaced Kac and Peterson's notation $(n,N,m)$ with $(m,\ell,N)$ of \cite{SW}.  Here our focus will be on double-sum evaluations.  A partial list of string functions from \cite[pp. 219--220]{KP} reads

Level 1:
\begin{subequations}
\begin{gather}
c_{10}^{10}=\eta(\tau)^{-1},\label{equation:KP-1-orig}
\end{gather}
\end{subequations}

Level 2:
\begin{subequations}
\begin{gather}
c_{11}^{11}=\eta(\tau)^{-2}\eta(2\tau),\label{equation:KP-2-orig}
\end{gather}
\end{subequations}

Level 4:
\begin{subequations}
\begin{gather}
c_{22}^{40}=\eta(\tau)^{-2}\eta(6\tau)^{-1}\eta(12\tau)^2,\label{equation:KP-4A-orig}\\
c_{40}^{40}-c_{04}^{40}=\eta(2\tau)^{-1},\label{equation:KP-4B-orig}
\end{gather}
\end{subequations}

Level 6:
\begin{subequations}
\begin{gather}
c_{51}^{33}=\eta(\tau)^{-3}\eta(2\tau)\eta(3\tau)\eta(6\tau)^{-1}\eta(12\tau),\label{equation:KP-6A-orig}\\
c_{51}^{51}+c_{15}^{51}=\eta(\tau)^{-3}\eta(2\tau)\eta(6\tau)^2\eta(12\tau)^{-1},\label{equation:KP-6B-orig}\\
c_{51}^{51}-c_{15}^{51}=\eta(\tau)^{-1},\label{equation:KP-6C-orig}
\end{gather}
\end{subequations}

Level 8:
\begin{subequations}
\begin{gather}
c_{62}^{44}=\eta(\tau)^{-3}\eta(2\tau)\eta(10\tau),\label{equation:KP-8A-orig}\\
c_{62}^{62}-c_{26}^{62}=\eta(\tau)^{-1}\eta(2\tau)^{-1}q^{1/10}
\prod_{n \not \equiv \pm 1 \pmod {5}}(1-q^{4n}),\label{equation:KP-8B-orig}
\end{gather}
\end{subequations}

Level 10:
\begin{subequations}
\begin{gather}
c_{55}^{73}=\eta(\tau)^{-3}\eta(2\tau)\eta(5\tau)^{-1}\eta(10\tau)^{2},\label{equation:KP-10A-orig}\\
c_{91}^{55}=\eta(\tau)^{-3}q^{29/40}\prod_{n \not \equiv \pm 1 \pmod {5}}(1-q^{2n})
\prod_{n \not \equiv \pm 2 \pmod {5}}(1-q^{3n}),\label{equation:KP-10B-orig}\\
c_{91}^{91}-c_{19}^{91}=\eta(\tau)^{-2}\eta(2\tau)q^{-1/15}\prod_{n \equiv \pm 1 \pmod {5}}(1-q^{4n})^{-1}.\label{equation:KP-10C-orig}
\end{gather}
\end{subequations}

In this paper we rewrite the double-sum form of the string functions (\ref{equation:KP-fullGlory}) using an expression found in \cite{SW}, where from \cite{SW} and \cite[p. 260]{KP} we have that
\begin{equation}
\eta^{3}(\tau)c_{N-m,m}^{N-\ell,\ell}
=q^{\frac{(\ell+1)^2}{4(N+2)}-\frac{m^2}{4N}}
\cdot f_{1,1+N,1}(q^{1+\tfrac{1}{2}(m+\ell)},q^{1-\tfrac{1}{2}(m-\ell)},q).\label{equation:gen-fabc-SW}
\end{equation}

\smallskip
Kac and Peterson appeal to modularity to prove the string function identities  \cite[p. 220]{KP}.  Specifically, they use the transformation law for string functions under the full modular group, together with the calculation of the first few terms in the Fourier expansions of the string functions.  They also take advantage of the fact that a modular form vanishing at cusps to sufficiently high order is zero.  There are many methods to compute string functions, see \cite[pp. 222--223]{KP} for a brief outline.  In our present work, we will take a different approach.

\smallskip
Once string function identities (\ref{equation:KP-1-orig})--(\ref{equation:KP-10C-orig}) have been written in terms of suitable Hecke-type double-sums, see (\ref{equation:KP-1-Hecke})--(\ref{equation:KP-10C-Hecke}), their modularity can be determined using results of \cite{HM}, see in particular the formulas found in Section \ref{section:double-sums}.   A simple example of more general results found in \cite{HM} reads
\begin{align}
f_{1,2,1}(x,y,q)&:=\Big ( \sum_{r,s\ge 0 }-\sum_{r,s<0}\Big )(-1)^{r+s}x^ry^sq^{\binom{r}{2}+2rs+\binom{s}{2}}
\label{equation:f121}\\
& \ =j(y;q)m\big (\frac{q^2x}{y^2},q^3,-1\big )+j(x;q)m\big (\frac{q^2y}{x^2},q^3,-1\big )\notag \\
&\ \ \ \ \ \ \  - \frac{yJ_3^3j(-x/y;q)j(q^2xy;q^3)}
{\overline{J}_{0,3}j(-qy^2/x;q^3)j(-qx^2/y;q^3)}.\notag
\end{align}
Such formulas were found by using a heuristic that relates Appell--Lerch functions to divergent partial theta functions, see \cite[Section 3]{HM} and \cite[Section 4]{Mort2021}.  

\smallskip
For an easy demonstration of (\ref{equation:f121}), we consider the case level $N=1$.  It is straightforward to see from (\ref{equation:KP-fullGlory}) that (\ref{equation:KP-1-orig}) is equivalent to showing:
\begin{gather}
f_{1,2,1}(q,q,q)=J_1^2.\label{equation:KP-1-Hecke}
\end{gather}
One sees from (\ref{equation:f121}) that
\begin{align*}
f_{1,2,1}(q,q,q)&=j(q;q)m(q,q^3,-1)+j(q;q)m(q,q^3,-1)
-\frac{1}{\overline{J}_{0,3}}\cdot \frac{qJ_{3}^3\overline{J}_{0,1}J_{4,3}}
{\overline{J}_{2,3}^2}\\
&=0+0+J_1^2=J_1^2,
\end{align*}
where we used the fact that $j(x;q)=0$ if and only if $x$ is an integral power of $q$ and elementary product rearrangements.  On the other hand, if we make a slight variation in the inputs, we have
\begin{equation*}
f_{1,2,1}(q,-q,q)=j(-q;q)m(q,q^3,-1)+0+0=2\overline{J}_{1,4}m(q,q^3,-1)= \overline{J}_{1,4}\cdot \phi(q),\label{equation:phi-lerch}
\end{equation*}
where $\phi(q)$ is a sixth-order mock theta function \cite{AH}. 

\smallskip
In general, one needs to be aware of the number of theta quotients that must be dealt with in simplifying the results.  Double-sums of the form $f_{1,n+1,1}$ require computing $n^2$ theta quotients, see in particular Theorem \ref{theorem:masterFnp}.   For example, one sees from (\ref{equation:KP-fullGlory}) that (\ref{equation:KP-10B-orig}) is equivalent to showing
\begin{equation}
f_{1,11,1}(q^4,q^3,q)=J_{4,10}J_{3,15},\label{equation:f1n1-100}
\end{equation}
which requires working with one-hundred theta quotients!  Looking for ways to simplify the process leads us to our main theorem (Theorem \ref{theorem:main-result}).

\smallskip
Double-sums of the form $f_{K+1,K+1,1}$  that appear in Theorem \ref{theorem:main-result} only have $K+1$ theta quotients that need to be summed.  Using our main theorem and its corollaries, the string function identities (\ref{equation:KP-2-orig})--(\ref{equation:KP-10C-orig}) are equivalent to the respective double-sum evaluations:  

\smallskip
Level 2:
\begin{subequations}
\begin{gather}
f_{2,2,1}(q^2,q,q)=J_1J_2,\label{equation:KP-2-Hecke}
\end{gather}
\end{subequations}

Level 4:
\begin{subequations}
\begin{gather}
q^{-1}f_{3,3,1}(q^2,1,q)=J_1\overline{J}_{6,24},\label{equation:KP-4A-Hecke}\\
f_{3,3,1}(-q^2,q,q)-qf_{3,3,1}(-q^4,q^3,q)=J_1J_{1,2},\label{equation:KP-4B-Hecke}
\end{gather}
\end{subequations}

Level 6:
\begin{subequations}
\begin{gather}
f_{4,4,1}(q^4,q^3,q)=J_2J_{3,12},\label{equation:KP-6A-Hecke}\\
f_{4,4,1}(q^3,q^2,q)+qf_{4,4,1}(q^5,q^4,q)=J_2J_{6,12},\label{equation:KP-6B-Hecke}\\
f_{4,4,1}(-q^3,q^2,q)-qf_{4,4,1}(-q^5,q^4,q)=J_1^2,\label{equation:KP-6C-Hecke}
\end{gather}
\end{subequations}

Level 8:
\begin{subequations}
\begin{gather}
f_{5,5,1}(q^5,q^4,q)=J_2J_{10},\label{equation:KP-8A-Hecke}\\
f_{5,5,1}(-q^4,q^3,q)-qf_{5,5,1}(-q^6,q^5,q)=J_{1,2}J_{8,20},\label{equation:KP-8B-Hecke}
\end{gather}
\end{subequations}

Level 10:
\begin{subequations}
\begin{gather}
q^{-1}f_{6,6,1}(q^5,1,q)=J_2\overline{J}_{5,20},\label{equation:KP-10A-Hecke}\\
f_{6,6,1}(q^6,q^4,q)=J_{4,10}J_{3,15},\label{equation:KP-10B-Hecke}\\
f_{6,6,1}(-q^4,q^2,q)-q^2f_{6,6,1}(-q^8,q^6,q)=\frac{J_2J_{20}}{J_{1}^2J_{4,20}}.\label{equation:KP-10C-Hecke}
\end{gather}
\end{subequations}

In particular our example in (\ref{equation:f1n1-100}) becomes (\ref{equation:KP-10B-Hecke}), which only requires working with six theta quotients, see Theorem \ref{theorem:main-acdivb}.

\smallskip
In \cite{HM} we demonstrated that the double-sum formulas give straightforward proofs of the classical mock theta function identities, and in particular the formulas give new proofs of the mock theta conjectures \cite{H1, H2}.  The focus here is to demonstrate the robustness of our approach in application to string functions.   We evaluate eight string functions using double-sum formulas.   Four double-sum computations are pieced together from previous results and four double-sum computations are new.   

\smallskip
In Section \ref{section:prelim}, we recall background information on theta functions, Appell--Lerch functions, and Hecke-type double-sums.   In Section \ref{section:classical-string-functions}, we discuss classical string function relations in the environment of Hecke-type double-sums.   In Section \ref{section:Level-gen-N}, we obtain (\ref{equation:gen-fabc-SW}), where the corrected proof comes from a sketch found in \cite[Example 1.3]{HM}.  In Section \ref{section:new-string-functions}, we prove our main result: Theorem \ref{theorem:main-result} and its two corollaries, which one could consider to be new string function relations.  In the remaining sections, we use Theorem \ref{theorem:main-result} to prove even-level string function identities.

\smallskip
In Section \ref{section:LevelN2}, we evaluate the level $N=2$ string function $c_{11}^{11}$ (\ref{equation:KP-2-orig}).  In Section \ref{section:LevelN4}, we evaluate the level $N=4$ string function $c_{40}^{40}-c_{04}^{40}$ (\ref{equation:KP-4B-orig}), where the proof is essentially from \cite{Mort2021}.  In Section \ref{section:LevelN6}, we evaluate the level $N=6$ string functions $c_{51}^{33}$ (\ref{equation:KP-6A-orig}) and $c_{51}^{51}+c_{15}^{51}$ (\ref{equation:KP-6B-orig}).   In Section \ref{section:LevelN6}, we also evaluate the level $N=6$ string function $c_{51}^{51}-c_{15}^{51}$ (\ref{equation:KP-6C-orig}), where the proof is essentially from \cite{Mort2021}.     In Section \ref{section:LevelN8}, we evaluate the level $N=8$ string function $c_{62}^{44}$ (\ref{equation:KP-8A-orig}), where the proof is essentially from \cite{HM}.      In Section \ref{section:LevelN10}, we evaluate the level $N=10$ string function $c_{91}^{55}$ (\ref{equation:KP-10B-orig}).

\smallskip
For other work on integral level string functions, we draw the reader's attention to Lepowsky and Primc's \cite{LP}.

\section{Preliminaries}\label{section:prelim}

\subsection{Theta functions}
We collect some frequently encountered product rearrangements:
\begin{subequations}
\begin{gather}
\overline{J}_{0,1}=2\overline{J}_{1,4}=\frac{2J_2^2}{J_1},  \overline{J}_{1,2}=\frac{J_2^5}{J_1^2J_4^2},   J_{1,2}=\frac{J_1^2}{J_2},   \overline{J}_{1,3}=\frac{J_2J_3^2}{J_1J_6},\notag \\
 J_{1,4}=\frac{J_1J_4}{J_2},   J_{1,6}=\frac{J_1J_6^2}{J_2J_3},   \overline{J}_{1,6}=\frac{J_2^2J_3J_{12}}{J_1J_4J_6}.\notag
\end{gather}
\end{subequations}
Following from the definitions are the following general identities:
{\allowdisplaybreaks \begin{subequations}
\begin{gather}
j(z;q)=\sum_{k=0}^{m-1}(-1)^k q^{\binom{k}{2}}z^k
j\big ((-1)^{m+1}q^{\binom{m}{2}+mk}z^m;q^{m^2}\big ),\label{equation:j-split}\\
j(q^n x;q)=(-1)^nq^{-\binom{n}{2}}x^{-n}j(x;q), \ \ n\in\mathbb{Z},\label{equation:j-elliptic}\\
j(x;q)=j(q/x;q)\label{equation:1.7},\\
j(x;q)={J_1}j(x;q^n)j(qx;q^n)\cdots j(q^{n-1}x;q^n)/{J_n^n} \ \ {\text{if $n\ge 1$,}}\label{equation:1.10}\\
j(x;-q)={j(x;q^2)j(-qx;q^2)}/{J_{1,4}},\label{equation:1.11}\\
j(x^n;q^n)={J_n}j(x;q)j(\zeta_nx;q)\cdots j(\zeta_n^{n-1}x;q^n)/{J_1^n},  \label{equation:1.12}
\end{gather}
\end{subequations}}%
if $n\ge 1$, $\zeta_n$ is a primitive $n$-th root of unity.

\smallskip
 A convenient form of the Weierstrass three-term relation for theta functions is,
\begin{proposition}\label{proposition:Weierstrass} For generic $a,b,c,d\in \mathbb{C}^*$
\begin{align*}
j(ac;q)&j(a/c;q)j(bd;q)j(b/d;q)\\
&=j(ad;q)j(a/d;q)j(bc;q)j(b/c;q)+b/c \cdot j(ab;q)j(a/b;q)j(cd;q)j(c/d;q).
\end{align*}
\end{proposition}

We collect several useful results about theta functions in terms of a proposition \cite{H1, H2}: 
\begin{proposition}   For generic $x,y,z\in \mathbb{C}^*$ 
{\allowdisplaybreaks \begin{subequations}
\begin{gather}
j(-x;q)j(y;q)-j(x;q)j(-y;q)=2xj(x^{-1}y;q^2)j(qxy;q^2),\label{equation:H1Thm1.2A}\\
j(-x;q)j(y;q)+j(x;q)j(-y;q)=2j(xy;q^2)j(qx^{-1}y;q^2).\label{equation:H1Thm1.2B}
\end{gather}
\end{subequations}}%
\end{proposition}

We finish the subsection by giving a few examples of elementary theta function identities and product rearrangements in action.
\begin{lemma} \label{lemma:f661-identities}  We have
\begin{equation}
\frac{J_{30}^3}{\overline{J}_{0,5}} 
\frac{J_{3,6}}{J_{3,30}}
- 2\cdot  \frac{J_{6,60}\overline{J}_{5,30}\overline{J}_{10,30}}{\overline{J}_{0,5}\overline{J}_{0,30}\overline{J}_{3,30}}\cdot \frac{J_{6}J_{60}}{J_{30}^4}
\cdot  J_{9,30}J_{21,30}J_{15,30}J_{5,30}
=0,\label{equation:f661-idA}
\end{equation}
and
\begin{equation}
4\cdot  \frac{J_{6,60}\overline{J}_{5,30}\overline{J}_{10,30}}{\overline{J}_{0,5}\overline{J}_{0,30}\overline{J}_{3,30}}
\cdot \frac{J_{6}J_{60}}{J_{30}^4}
\cdot J_{16,30}J_{20,30}J_{26,30} J_{10,30}
=J_{4,10}J_{3,15}.\label{equation:f661-idB}
\end{equation}
\end{lemma}

\begin{proof}[Proof of Lemma \ref{lemma:f661-identities}]

For the first summand in (\ref{equation:f661-idA}), we apply (\ref{equation:1.10}) to $J_{3,6}$ to obtain
\begin{align*}
\frac{J_{30}^3J_{3,6}}{\overline{J}_{0,5}J_{3,30}} 
&=\frac{J_{30}^3}{\overline{J}_{0,5}J_{3,30}} 
\cdot J_{3,30}J_{9,30}J_{15,30}J_{21,30}J_{27,30} \cdot \frac{J_6}{J_{30}^5}\\
&=\frac{J_{3,30}J_{9,30}J_{15,30}J_{21,30}J_6}{\overline{J}_{0,5}J_{30}^2}. 
\end{align*}
For the second summand, we begin by applying (\ref{equation:1.12}) to $J_{6,60}$ to have
{\allowdisplaybreaks \begin{align*}
2\cdot & \frac{J_{6,60}\overline{J}_{5,30}\overline{J}_{10,30}}{\overline{J}_{0,5}\overline{J}_{0,30}\overline{J}_{3,30}}\cdot \frac{J_{6}J_{60}}{J_{30}^4}
\cdot  J_{9,30}J_{21,30}J_{15,30}J_{5,30}\\
&=2\cdot  \frac{\overline{J}_{5,30}\overline{J}_{10,30}}{\overline{J}_{0,5}\overline{J}_{0,30}\overline{J}_{3,30}}\cdot J_{3,30}\overline{J}_{3,30}\cdot \frac{J_{60}}{J_{30}^2}\cdot \frac{J_{6}J_{60}}{J_{30}^4}
\cdot  J_{9,30}J_{21,30}J_{15,30}J_{5,30}\\
&= \frac{\overline{J}_{5,30}\overline{J}_{10,30}}{\overline{J}_{0,5}}\cdot J_{3,30}\cdot \frac{J_{6}}{J_{30}^5}
\cdot  J_{9,30}J_{21,30}J_{15,30}J_{5,30},
\end{align*}}%
where we have used the product rearrangement $\overline{J}_{0,1}=2J_2^2/J_1$ and simplified.  Applying (\ref{equation:1.12}) to the product $J_{5,30}\overline{J}_{5,30}$ yields
{\allowdisplaybreaks \begin{align*}
2\cdot & \frac{J_{6,60}\overline{J}_{5,30}\overline{J}_{10,30}}{\overline{J}_{0,5}\overline{J}_{0,30}\overline{J}_{3,30}}\cdot \frac{J_{6}J_{60}}{J_{30}^4}
\cdot  J_{9,30}J_{21,30}J_{15,30}J_{5,30}\\
&= \frac{\overline{J}_{10,30}}{\overline{J}_{0,5}}\cdot J_{3,30}\cdot \frac{J_{6}}{J_{30}^5}
\cdot  J_{9,30}J_{21,30}J_{15,30}J_{10,60}\cdot\frac{J_{30}^2}{J_{60}}\\
&= \frac{J_{3,30}}{\overline{J}_{0,5}} \cdot \frac{J_{6}}{J_{30}^5}
\cdot  J_{9,30}J_{21,30}J_{15,30}
\cdot \frac{J_{10}J_{60}^2}{J_{20}J_{30}}
\cdot\frac{J_{30}^2}{J_{60}}\cdot \frac{J_{20}J_{30}^2}{J_{10}J_{60}}\\
&= \frac{J_{3,30}J_{9,30}J_{15,30}J_{21,30}J_{6}}{\overline{J}_{0,5}J_{30}^2}, 
\end{align*}}%
where for the second equality we used the product rearrangements for $\overline{J}_{1,3}$ and $J_{1,6}$, and for the third equality we simplified.

For the second identity (\ref{equation:f661-idB}), we use the product rearrangements for $\overline{J}_{0,1}$, $\overline{J}_{1,6}$, and $\overline{J}_{1,3}$ and simplify to have
{\allowdisplaybreaks \begin{align*}
4\cdot & \frac{J_{6,60}\overline{J}_{5,30}\overline{J}_{10,30}}{\overline{J}_{0,5}\overline{J}_{0,30}\overline{J}_{3,30}}
\cdot \frac{J_{6}J_{60}}{J_{30}^4}
\cdot J_{16,30}J_{20,30}J_{26,30} J_{10,30}\\
&=  \frac{J_{6,60}\overline{J}_{5,30}\overline{J}_{10,30}}{\overline{J}_{3,30}}
\cdot \frac{J_{6}}{J_{30}^3}
\cdot J_{16,30}J_{26,30} \cdot \frac{J_{5}}{J_{60}}\\
&=  \frac{J_{6,60}}{\overline{J}_{3,30}}
\cdot \frac{J_{6}}{J_{30}^2}
\cdot J_{16,30}J_{26,30} \cdot \frac{J_{10}J_{15}}{J_{60}}\\
&=J_{3,30}
\cdot \frac{J_{6}}{J_{30}^4}
\cdot J_{16,30}J_{26,30} \cdot J_{10}J_{15},
\end{align*}}%
where for the last equality we used (\ref{equation:1.12}).  Identity $(\ref{equation:1.10})$, gives us
\begin{equation*}
J_{6,10}=J_{6,30}J_{16,30}J_{26,30}\cdot \frac{J_{10}}{J_{30}^3} \textup{ and } J_{3,15}=J_{3,30}J_{18,30}\cdot \frac{J_{15}}{J_{30}^2}.
\end{equation*}
Hence we can rewrite the right-hand side as
\begin{equation*}
J_{3,30} \cdot \frac{J_{6}}{J_{30}^4}
\cdot J_{16,30}J_{26,30} \cdot J_{10}J_{15}
=J_{6,10}J_{3,15}\cdot \frac{J_6J_{30}}{J_{6,30}J_{18,30}}
=J_{4,10}J_{3,15},
\end{equation*}
where we have used the product rearrangement $J_{1,5}J_{2,5}=J_{1}J_{5}$ and (\ref{equation:1.7}).
\end{proof}

\subsection{Appell--Lerch functions}\label{section:prop-mxqz}

The Appell--Lerch function satisfies several functional equations and identities \cite{HM, Zw}:

\begin{proposition}  For generic $x,z\in \mathbb{C}^*$
{\allowdisplaybreaks \begin{subequations}
\begin{gather}
m(x,q,z)=m(x,q,qz),\label{equation:mxqz-fnq-z}\\
m(x,q,z)=x^{-1}m(x^{-1},q,z^{-1}),\label{equation:mxqz-flip}\\
m(qx,q,z)=1-xm(x,q,z),\label{equation:mxqz-fnq-x}\\
m(x,q,z)=m(x,q,x^{-1}z^{-1}), \label{equation:mxqz-flip-xz}\\
m(x,q,z_1)-m(x,q,z_0)=\frac{z_0J_1^3j(z_1/z_0;q)j(xz_0z_1;q)}{j(z_0;q)j(z_1;q)j(xz_0;q)j(xz_1;q)}.
\label{equation:changing-z-theorem}
\end{gather}
\end{subequations}}
\end{proposition}
Some simple evaluations of the Appell--Lerch function follow.
\begin{corollary} \label{corollary:mxqz-eval} We have
\begin{align}
m(q,q^2,-1)&=1/2,\label{equation:mxqz-eval-a}\\
m(-1,q^2,q)&=0. \label{equation:mxqz-eval-b}
\end{align}
\end{corollary}

\subsection{Hecke-type double-sums}\label{section:double-sums}
We recall a few basic properties of Hecke-type double-sums.  We have a proposition and a corollary

\begin{proposition} \cite[Proposition $6.3$]{HM}  For $x,y\in\mathbb{C}^{*}$ and $R,S\in\mathbb{Z}$
{\allowdisplaybreaks \begin{align}
f_{a,b,c}(x,y,q)&=(-x)^{R}(-y)^Sq^{a\binom{R}{2}+bRS+c\binom{S}{2}}f_{a,b,c}(q^{aR+bS}x,q^{bR+cS}y,q) \label{equation:f-shift}\\
&\ \ \ \ +\sum_{m=0}^{R-1}(-x)^mq^{a\binom{m}{2}}j(q^{mb}y;q^c)+\sum_{m=0}^{S-1}(-y)^mq^{c\binom{m}{2}}j(q^{mb}x;q^a).\notag
\end{align}}%
\end{proposition}
When $b<a$ we adopt the summation convention that
\begin{equation}
\sum_{n=a}^{b}c_{n}:=-\sum_{n=b+1}^{a-1}c_{n},
\end{equation}
which has the useful consequence
\begin{equation}
\sum_{n=0}^{-1}c_{n}=-\sum_{n=0}^{-1}c_{n}=0.
\end{equation}

Usually when (\ref{equation:f-shift}) is used in this paper, the two summands of theta functions in the second row are always equal to zero.  In such cases, it is due to the fact that $j(x;q)=0$ if and only if $x$ is an integral power of $q$.  The only exception occurs in the last subsection of Section \ref{section:LevelN6}. 

\begin{corollary} \cite[Corollary $6.4$]{HM} \label{corollary:fabc-funceqnspecial} We have two simple specializations:
\begin{align}
f_{a,b,c}(x,y,q) =&-yf_{a,b,c}(q^bx,q^cy,q)+j(x;q^a),\label{equation:fabc-fnq-1}\\
f_{a,b,c}(x,y,q) =&-xf_{a,b,c}(q^ax,q^by,q)+j(y;q^c).\label{equation:fabc-fnq-2}
\end{align}
\end{corollary}
We also have the property \cite[$(6.2)$]{HM}:
\begin{equation}
f_{a,b,c}(x,y,q)=-\frac{q^{a+b+c}}{xy}f_{a,b,c}(q^{2a+b}/x,q^{2c+b}/y,q).\label{equation:f-flip}
\end{equation}

One can determine the theta function expressions of string functions using results of \cite{HM}.  However, there are two recurring issues that must be kept in mind.  The first issue is the number of theta quotients that must be computed, and the second issue is potential singularities.

To state our results, we introduce the useful
\begin{align}
g_{1,b,1}(x,y,q,z_1,z_0)
&:=j(y;q)m\Big (q^{\binom{b+1}{2}-1}x(-y)^{-b},q^{b^2-1},z_1\Big )\label{equation:mdef-2}\\
&\ \ \ \ \ + j(x;q)m\Big (q^{\binom{b+1}{2}-1}y(-x)^{-b},q^{b^2-1},z_0\Big ).\notag
\end{align}

In \cite[Theorem 1.3]{HM}, we specialize $n=1$, to have 
\begin{theorem}  \label{theorem:masterFnp} Let $p$ be a positive integer.  For generic $x,y\in \mathbb{C}^*$
\begin{align*}
f_{1,p+1,1}(x,y,q)=g_{1,p+1,1}(x,y,q,-1,-1)+\frac{1}{\overline{J}_{0,p(2+p)}}\cdot \theta_{p}(x,y,q),
\end{align*}
where
\begin{align*}
&\theta_{p}(x,y,q):=\sum_{r=0}^{p-1}\sum_{s=0}^{p-1}q^{\binom{r}{2}+(1+p) (r) (s+1 )+\binom{s+1}{2}}  (-x)^{r}(-y)^{s+1}\notag\\
 & \cdot   \frac{J_{p^2(2+p)}^3j(-q^{p(s-r)}x/y;q^{p^2})j(q^{p(2+p)(r+s)+p(1+p)}x^py^p;q^{p^2(2+p)})}{j(q^{p(2+p)r+p(1+p)/2}(-y)^{1+p}/(-x);q^{p^2(2+p)})j(q^{p(2+p)s+p(1+p)/2}(-x)^{1+p}/(-y);q^{p^2(2+p)})}.\notag
\end{align*}
\end{theorem}

The specialization for $p=1$ will be of importance.  It is just (\ref{equation:f121}):
\begin{corollary} We have
\begin{align}
&\Big (\sum_{r,s\ge 0}-\sum_{r,s<0}\Big )(-1)^{r+s}x^ry^sq^{\binom{r}{2}+2rs+\binom{s}{2}}\\
&\ \ \ \ \ =j(y;q)m\Big (\frac{q^2x}{y^2},q^3,-1\Big )+j(x;q)m\Big (\frac{q^2y}{x^2},q^3,-1\Big )
- \frac{yJ_{3}^3j(-x/y;q)j(q^2xy;q^3)}{\overline{J}_{0,3}j(-qy^2/x,-qx^2/y;q^3)}.\notag 
\end{align}
\end{corollary}

\noindent For another useful result, we specialize \cite[Theorem 1.4]{HM} to $a=b=n$, $c=1$.  
\begin{theorem} \label{theorem:main-acdivb} Let $n$ be a positive integer.  Then
\begin{align*}
& f_{n,n,1}(x,y,q)=h_{n,n,1}(x,y,q,-1,-1)-\frac{1}{\overline{J}_{0,n-1}\overline{J}_{0,n^2-n}}\cdot \theta_{n}(x,y,q),
\end{align*}
where 
\begin{align*}
h_{n,n,1}(x,y,q,z_1,z_0):&=j(x;q^n)m\Big( -q^{n-1}yx^{-1},q^{n-1},z_1 \Big )\\
& \ \ \ \ \ \ +j(y;q)m\Big( q^{\binom{n}{2}}x(-y)^{-n},q^{n^2-n},z_0 \Big ),
\end{align*}
and
\begin{align*}
\theta_{n}(x,y,q):=&\sum_{d=0}^{n-1}
q^{(n-1)\binom{d+1}{2}}j\big (q^{(n-1)(d+1)}y;q^{n}\big )  j\big (-q^{n(n-1)-(n-1)(d+1)}xy^{-1};q^{n(n-1)}\big ) \\
& \ \ \ \ \ \cdot \frac{J_{n(n-1)}^3j\big (q^{ \binom{n}{2}+(n-1)(d+1)}(-y)^{1-n};q^{n(n-1)}\big )}
{j\big (-q^{\binom{n}{2}}x(-y)^{-n};q^{n(n-1)})j(q^{(n-1)(d+1)}x^{-1}y;q^{n(n-1)}\big )}.
\end{align*}
\end{theorem}

Theorem \ref{theorem:main-acdivb} has the following specializations.

\begin{corollary}  \label{corollary:f221-expansion}  We have
{\allowdisplaybreaks \begin{align}
f_{2,2,1}(x,y,q)&=h_{2,2,1}(x,y,q,-1,-1)\label{equation:f221-id}\\
&\ \ \ \ -\sum_{d=0}^{1}
\frac{q^{\binom{d+1}{2}}j\big (q^{1+d}y;q^{2}\big )  j\big (-q^{1-d}x/y;q^{2}\big ) 
J_{2}^3j\big (-q^{2+d}/y;q^{2}\big )}
{4\overline{J}_{1,4}\overline{J}_{2,8}j\big (-qx/y^{2};q^2\big )j\big (q^{1+d}y/x;q^2\big )},\notag
\end{align}}%
where
\begin{align}
h_{2,2,1}(x,y,q,-1,-1)=j(x;q^2)m(-qx^{-1}y,q,-1)+j(y;q)m(qxy^{-2},q^2,-1).\label{equation:g221-id}
\end{align}
\end{corollary}

\begin{corollary}  \label{corollary:f331-expansion}  We have
{\allowdisplaybreaks \begin{align}
f_{3,3,1}(x,y,q)&=h_{3,3,1}(x,y,q,-1,-1)\label{equation:f331-id}\\
&\ \ \ \ -\sum_{d=0}^{2}
\frac{q^{d(d+1)}j\big (q^{2+2d}y;q^{3}\big )  j\big (-q^{4-2d}x/y;q^{6}\big ) J_{6}^3j\big (q^{5+2d}/y^{2};q^{6}\big )}
{4\overline{J}_{2,8}\overline{J}_{6,24}j\big (q^{3}x/y^{3};q^6\big )j\big (q^{2+2d}y/x;q^6\big )},\notag
\end{align}}%
where
\begin{align}
h_{3,3,1}(x,y,q,-1,-1)=j(x;q^3)m(-q^2x^{-1}y,q^2,-1)+j(y;q)m(-q^3xy^{-3},q^6,-1).\label{equation:g331-id}
\end{align}
\end{corollary}

\begin{corollary}  \label{corollary:f441-expansion} We have
\begin{align}
f_{4,4,1}(x,y,q)&=h_{4,4,1}(x,y,q,-1,-1)\label{equation:f441}\\
&\ \ \ \ \ -\sum_{d=0}^3\frac{q^{3\binom{d+1}{2}}j(q^{3+3d}y;q^4)j(-q^{9-3d}x/y;q^{12})J_{12}^3j(-q^{9+3d}/y^3;q^{12})}{\overline{J}_{0,3}\overline{J}_{0,12}j(-q^{6}x/y^4;q^{12})j(q^{3+3d}y/x;q^{12})},\notag
\end{align}
where
\begin{equation}
h_{4,4,1}(x,y,q,-1,-1)=j(x;q^4)m\big ( -q^3y/x,q^3,-1 \big )+j(y;q) m\big ( q^6 x/y^4, q^{12}, -1 \big ).\label{equation:h441}
\end{equation}
\end{corollary}

\begin{corollary} \label{corollary:f551-expansion} We have
\begin{align}
f_{5,5,1}(x,y,q)&=h_{5,5,1}(x,y,q,-1,-1)\label{equation:f551}\\
&\ \ \ \ \ -\sum_{d=0}^{4}
q^{2d(d+1)}\frac{j\big (q^{4+4d}y;q^{5}\big )  j\big (-q^{16-4d}xy^{-1};q^{20}\big )J_{20}^3j\big (q^{14+4d}y^{-4};q^{20}\big )}
{\overline{J}_{0,4}\overline{J}_{0,20}j\big (q^{10}xy^{-5};q^{20})j(q^{4+4d}x^{-1}y;q^{20}\big )},\notag
\end{align}
where
\begin{equation}
h_{5,5,1}(x,y,q,z_1,z_0)
=j(x;q^5)m(-q^4x^{-1}y,q^4,z_1)+j(y;q)m(-q^{10}xy^{-5},q^{20},z_0).\label{equation:h551}
\end{equation}
\end{corollary}

\begin{corollary}  \label{corollary:f661-expansion}  We have
\begin{align}
f_{6,6,1}(x,y,q)&=h_{6,6,1}(x,y,q,-1,-1)\label{equation:f661}\\
& \ \ \ \ \ - \sum_{d=0}^{5}
  \frac{q^{5\binom{d+1}{2}}J_{30}^3j\big (q^{5d+5}y;q^{6}\big )  j\big (-q^{25-5d}xy^{-1};q^{30}\big )j\big (-q^{ 20+5d}y^{-5};q^{30}\big )}
{\overline{J}_{0,5}\overline{J}_{0,30}j\big (-q^{15}xy^{-6};q^{30})j(q^{5d+5}x^{-1}y;q^{30}\big )},\notag
\end{align}
where 
\begin{equation}
h_{6,6,1}(x,y,q,z_1,z_0):=j(x;q^6)m( -q^{5}yx^{-1},q^{5},z_1 )
+j(y;q)m( q^{15}xy^{-6},q^{30},z_0  ).
\label{equation:h661}
\end{equation}
\end{corollary}

We finish with a few example computations that can be found in the literature.
\begin{lemma}\label{lemma:f441-f331-evals}  We have
\begin{subequations}
\begin{gather}
f_{5,5,1}(q^5,q^4,q)=J_{2}J_{10},\\
f_{4,4,1}(-q^{5},q^{3},q)-q^{-1}f_{4,4,1}(-q^{3},q,q)
=-q^{-1}J_1^2,\\
f_{3,3,1}(-q^{4},q^{3},q)-q^{-1}f_{3,3,1}(-q^{2},q,q)
=-q^{-1}J_1J_{1,2}.
\end{gather}
\end{subequations}
\end{lemma}
\begin{proof}[Proof of Lemma \ref{lemma:f441-f331-evals}]  All three identities have been shown using the appropriate specializations of Theorem \ref{theorem:main-acdivb}.  The first can be found in \cite[Section 7]{HM} and the last two in  \cite[Lemma 3.11]{Mort2021}.
\end{proof}


\section{Classical string function relations and Hecke-type double-sums}\label{section:classical-string-functions}
In this section we use the environment of Hecke-type double-sums and their functional equations such as (\ref{equation:f-shift}) and (\ref{equation:f-flip}) to show that
\begin{equation}
S_{m,\ell}^{N}(q):=\frac{q^{-\tfrac{1}{8}+\tfrac{(\ell+1)^2}{4(N+2)}-\tfrac{m^2}{4N}}}{J_{1}^3}
f_{1,1+N,1}(q^{1+\tfrac{1}{2}(m+\ell)},q^{1-\tfrac{1}{2}(m-\ell)},q).
\end{equation} 
satisfies the same classical symmetries for string functions.  We recall the classical identities \cite[(3.4), (3.5)]{SW}:
{\allowdisplaybreaks \begin{align}
S_{m,\ell}^{N}(q)&=S_{-m,\ell}^{N}(q),\label{equation:SW-3pt4}\\
S_{m,\ell}^{N}(q)&=S_{2N-m,\ell}^{N}(q),\label{equation:SW-3pt5A}\\
S_{m,\ell}^{N}(q)&=S_{N-m,N-\ell}^{N}(q). \label{equation:SW-3pt5B}
\end{align}}%

\begin{proof}[Proof for Identity (\ref{equation:SW-3pt4}):]
We have
{\allowdisplaybreaks \begin{align*}
S_{m,\ell}^N(q)
&=\frac{q^{-\tfrac{1}{8}+\tfrac{(\ell+1)^2}{4(N+2)}-\tfrac{m^2}{4N}}}{J_1^3}
\cdot f_{1,1+N,1}(q^{1+\tfrac{1}{2}(m+\ell)},q^{1-\tfrac{1}{2}(m-\ell)},q)\\
&=\frac{q^{-\tfrac{1}{8}+\tfrac{(\ell+1)^2}{4(N+2)}-\tfrac{m^2}{4N}}}{J_1^3}
\cdot f_{1,1+N,1}(q^{1-\tfrac{1}{2}(-m-\ell)},q^{1+\tfrac{1}{2}(-m+\ell)},q)\\
&=\frac{q^{-\tfrac{1}{8}+\tfrac{(\ell+1)^2}{4(N+2)}-\tfrac{(-m)^2}{4N}}}{J_1^3}
\cdot f_{1,1+N,1}(q^{1+\tfrac{1}{2}(-m+\ell)},q^{1-\tfrac{1}{2}(-m-\ell)},q)\\
&=S_{-m,\ell}^N(q),
\end{align*}}%
where symmetry allows us to write
\begin{equation*}
f_{1,1+N,1}(x,y,q)=f_{1,1+N,1}(y,x,q).\qedhere 
\end{equation*}
\end{proof}

\begin{proof}[Proof for Identity (\ref{equation:SW-3pt5A}):]
Specializing (\ref{equation:f-shift}) with $(R,S)=(1,-1)$ gives
{\allowdisplaybreaks \begin{align*}
f_{1,1+N,1}(q^{1+\tfrac{1}{2}(m+\ell)},q^{1-\tfrac{1}{2}(m-\ell)},q)
&=q^{m-N}
f_{1,1+N,1}(q^{-N+1+\tfrac{1}{2}(m+\ell)},q^{N+1-\tfrac{1}{2}(m-\ell)},q)\\
&=q^{m-N}
f_{1,1+N,1}(q^{1-\tfrac{1}{2}(2N-m-\ell)},q^{1+\tfrac{1}{2}(2N-m+\ell)},q)\\
&=q^{m-N}
f_{1,1+N,1}(q^{1+\tfrac{1}{2}(2N-m+\ell)},q^{1-\tfrac{1}{2}(2N-m-\ell)},q).
\end{align*}}%
As a result,
{\allowdisplaybreaks \begin{align*}
S_{m,\ell}^{N}&=\frac{q^{-\tfrac{1}{8}+\tfrac{(\ell+1)^2}{4(N+2)}-\tfrac{m^2}{4N}}}{J_{1}^3}f_{1,1+N,1}(q^{1+\tfrac{1}{2}(m+\ell)},q^{1-\tfrac{1}{2}(m-\ell)},q)\\
&=\frac{q^{-\tfrac{1}{8}+\tfrac{(\ell+1)^2}{4(N+2)}-\tfrac{m^2}{4N}+m-N}}{J_{1}^3}
f_{1,1+N,1}(q^{1+\tfrac{1}{2}(2N-m+\ell)},q^{1-\tfrac{1}{2}(2N-m-\ell)},q)\\
&=\frac{q^{-\tfrac{1}{8}+\tfrac{(\ell+1)^2}{4(N+2)}-\tfrac{(2N-m)^2}{4N}}}{J_{1}^3}
f_{1,1+N,1}(q^{1+\tfrac{1}{2}(2N-m+\ell)},q^{1-\tfrac{1}{2}(2N-m-\ell)},q)\\
&=S_{2N-m,\ell}^{N}.\qedhere
\end{align*}}%
\end{proof}

\begin{proof}[Proof for Identity (\ref{equation:SW-3pt5B}):]
Specializing (\ref{equation:f-shift}) with $(R,S)=(1,0)$ gives
{\allowdisplaybreaks \begin{align*}
f_{1,1+N,1}&(q^{1+\tfrac{1}{2}(m+\ell)},q^{1-\tfrac{1}{2}(m-\ell)},q)\\
&=-q^{1+\tfrac{1}{2}(m+\ell)}
\cdot f_{1,1+N,1}(q^{2+\tfrac{1}{2}(m+\ell)},q^{2+N-\tfrac{1}{2}(m-\ell)},q)\\
&=q^{1+\tfrac{1}{2}(m+\ell)}\cdot
q^{-1-\ell}
\cdot f_{1,1+N,1}(q^{1+N-\tfrac{1}{2}(m+\ell)},q^{1+\tfrac{1}{2}(m-\ell)},q)\\
&=q^{\tfrac{1}{2}(m-\ell)}
\cdot f_{1,1+N,1}(q^{1+\tfrac{1}{2}(N-m+N-\ell)},q^{1-\tfrac{1}{2}(N-m-(N-\ell))},q),
\end{align*}}%
where for the second equality we used (\ref{equation:f-flip}).  Hence
{\allowdisplaybreaks \begin{align*}
S_{m,\ell}^{N}
&=\frac{q^{-\tfrac{1}{8}+\tfrac{(\ell+1)^2}{4(N+2)}-\tfrac{m^2}{4N}}}{J_{1}^3}
\cdot f_{1,1+N,1}(q^{1+\tfrac{1}{2}(m+\ell)},q^{1-\tfrac{1}{2}(m-\ell)},q)\\
&=\frac{q^{-\tfrac{1}{8}+\tfrac{(\ell+1)^2}{4(N+2)}-\tfrac{m^2}{4N}+\tfrac{1}{2}(m-\ell)}}{J_{1}^3}
\cdot f_{1,1+N,1}(q^{1+\tfrac{1}{2}(N-m+N-\ell)},q^{1-\tfrac{1}{2}(N-m-(N-\ell))},q)\\
&=\frac{q^{-\tfrac{1}{8}+\tfrac{(N-\ell+1)^2}{4(N+2)}-\tfrac{(N-m)^2}{4N}}}{J_{1}^3}
\cdot f_{1,1+N,1}(q^{1+\tfrac{1}{2}(N-m+N-\ell)},q^{1-\tfrac{1}{2}(N-m-(N-\ell))},q)\\
&=S_{N-m,N-\ell}^{N}.\qedhere
\end{align*}}%
\end{proof}


\section{Deriving the general integral-level $N$ string function}\label{section:Level-gen-N}  

We derive (\ref{equation:gen-fabc-SW}).  We recall the notation from \cite{SW} that $m, \ N\in\mathbb{N}$,  $\ell \in \{ 0,1,2,\dots, N\}$, $m\equiv \ell \pmod 2$, where $N$ is the level.  In \cite[p. 236]{SW}, see also \cite[(3.17)]{DQ}, one finds the Hecke-type  form for the general integral-level string function:
\begin{proposition} We have
\begin{equation}
C_{m,\ell}^N(q)=S_{m,\ell}^{N}(q),
\end{equation}
where
\begin{equation}
S_{m,\ell}^{N}(q):=\frac{q^{-\tfrac{1}{8}+\tfrac{(\ell+1)^2}{4(N+2)}-\tfrac{m^2}{4N}}}{J_{1}^3}
f_{1,1+N,1}(q^{1+\tfrac{1}{2}(m+\ell)},q^{1-\tfrac{1}{2}(m-\ell)},q).
\end{equation}
\end{proposition}
We begin with an identity from \cite{SW}:
 \begin{equation}
C_{m,\ell}^N(q)=\frac{q^{s(m,\ell,N)}}{(q)_{\infty}^3}\Big \{ \sum_{\substack{j\ge 1\\ k\le 0}} - \sum_{\substack{j\le 0 \\ k\ge 1}} \Big \}(-1)^{k-j}q^{\binom{k-j}{2}-Njk+\tfrac{1}{2}k(m-\ell)+\tfrac{1}{2}j(m+\ell)}.
\end{equation}
where from (\ref{equation:s-def}):
\begin{equation}
s(m,\ell,N):=-\tfrac{1}{8}+\tfrac{(\ell+1)^2}{4(N+2)}-\tfrac{m^2}{4N}.
\end{equation}

We rewrite the above double-sum:
{\allowdisplaybreaks \begin{align*}
C_{m,\ell}^N(q)&=\frac{q^{s(m,\ell,N)}}{(q)_{\infty}^3} \sum_{\substack{j\ge 1\\ k\le 0}} (-1)^{k-j}q^{\binom{k-j}{2}-Njk+\tfrac{1}{2}k(m-\ell)+\tfrac{1}{2}j(m+\ell)} \\
& \ \ \ \ \ - \frac{q^{s(m,\ell,N)}}{(q)_{\infty}^3} \sum_{\substack{j\le 0 \\ k\ge 1}} (-1)^{k-j}q^{\binom{k-j}{2}-Njk+\tfrac{1}{2}k(m-\ell)+\tfrac{1}{2}j(m+\ell)} \\
&=\frac{q^{s(m,\ell,N)}}{(q)_{\infty}^3} \sum_{\substack{j\ge 1\\ k\ge 0}} (-1)^{k+j}q^{\binom{k+j+1}{2}+Njk-\tfrac{1}{2}k(m-\ell)+\tfrac{1}{2}j(m+\ell)} \\
& \ \ \ \ \ - \frac{q^{s(m,\ell,N)}}{(q)_{\infty}^3} \sum_{\substack{j\le 0 \\ k< 0}} (-1)^{k+j}q^{\binom{k+j+1}{2}+Njk-\tfrac{1}{2}k(m-\ell)+\tfrac{1}{2}j(m+\ell)} \\
&=\frac{q^{s(m,\ell,N)}}{(q)_{\infty}^3} \sum_{\substack{j\ge 0\\ k\ge 0}} (-1)^{k+j+1}q^{\binom{k+j+2}{2}+N(j+1)k-\tfrac{1}{2}k(m-\ell)+\tfrac{1}{2}(j+1)(m+\ell)} \\
& \ \ \ \ \ - \frac{q^{s(m,\ell,N)}}{(q)_{\infty}^3} \sum_{\substack{j< 0 \\ k< 0}} (-1)^{k+j+1}q^{\binom{k+j+2}{2}+N(j+1)k-\tfrac{1}{2}k(m-\ell)+\tfrac{1}{2}(j+1)(m+\ell)} \\
&=\frac{q^{s(m,\ell,N)}}{(q)_{\infty}^3} \Big ( \sum_{\substack{j\ge 0\\ k\ge 0}} - \sum_{\substack{j< 0 \\ k< 0}} \Big )  (-1)^{k+j+1}q^{\binom{k+j+2}{2}+N(j+1)k-\tfrac{1}{2}k(m-\ell)+\tfrac{1}{2}(j+1)(m+\ell)} \\
&=-\frac{q^{s(m,\ell,N)}}{(q)_{\infty}^3} q^{1+\frac{1}{2}(m+\ell)}
f_{1,N+1,1}(q^{2+N-\frac{1}{2}(m-\ell)}, q^{2+\frac{1}{2}(m+\ell)},q),
\end{align*}}%
where we have used $k\rightarrow -k$, $j\rightarrow j+1$.  We recall the identity (\ref{equation:fabc-fnq-1})
\begin{equation*}
f_{a,b,c}(x,y,q)=-yf_{a,b,c}(q^bx,q^cy,q)+j(x;q^a),
\end{equation*}
and the fact that $j(z;q)=0$ if and only if $z=q^{n}$, $n\in\mathbb{Z}$.   This gives
{\allowdisplaybreaks \begin{align*}
C_{m,\ell}^N(q)
&=\frac{q^{s(m,\ell,N)}}{(q)_{\infty}^3}\Big( 
f_{1,N+1,1}(q^{1-\frac{1}{2}(m-\ell)}, q^{1+\frac{1}{2}(m+\ell)},q) - j(q^{1+\frac{1}{2}(m+\ell)};q)\Big ) \\
&=\frac{q^{s(m,\ell,N)}}{J_1^3}\cdot f_{1,1+N,1}(q^{1-\tfrac{1}{2}(m-\ell)},q^{1+\tfrac{1}{2}(m+\ell)},q)\\
&=\frac{q^{s(m,\ell,N)}}{J_1^3}\cdot f_{1,1+N,1}(q^{1+\tfrac{1}{2}(m+\ell)},q^{1-\tfrac{1}{2}(m-\ell)},q)\\
&=S_{m,\ell}^{N}(q),
\end{align*}}%
for the second equality the condition $m\equiv \ell \pmod 2$ forces that
\begin{equation}
j(q^{1+\frac{1}{2}(m+\ell)};q)=0.
\end{equation}
For the third equality, the $x$ and $y$ in $f_{1,N+1,1}(x,y,q)$ can be swapped because of symmetry in the definition of (\ref{equation:fabc-def}).


\section{Proof of the main theorem and its two corollaries}\label{section:new-string-functions}

The driving force of our proof is the following proposition.

\begin{proposition} \label{proposition:f1N1-fNN1}We have
\begin{align}
&f_{1,2K+1,1}(q^{d},q^{e},q)\pm q^{\frac{K+d+e}{2}}f_{1,2K+1,1}(q^{1+K+d},q^{1+K+e},q)\label{equation:KP-disguise}\\
&\ \ \ =f_{K+1,K+1,1}(\mp q^{(K+d+e)/2},q^{d},q)
 \mp q^{(K+2-d-e)/2}f_{K+1,K+1,1}(\mp q^{2+(3K-d-e)/2},q^{K+2-e},q).\notag
\end{align}
\end{proposition}

\begin{proof}[Proof of Proposition \ref{proposition:f1N1-fNN1}] We have
{\allowdisplaybreaks \begin{align*}
&f_{1,2K+1,1}(q^{d},q^{e},q)+q^{\frac{K+d+e}{2}}f_{1,2K+1,1}(q^{1+K+d},q^{1+K+e},q)\\
&=\sum_{\substack{u,v\\u\equiv v \mod 2}}\sg(u,v)(-1)^{\frac{u-v}{2}}q^{\frac{1}{8}u^2+\frac{2K+1}{4}uv+\frac{1}{8}v^2+\frac{2d-1}{4}u+\frac{2e-1}{4}v}\\
&=\Big ( \sum_{\substack{n+j\ge0\\n-j\ge 0}}- \sum_{\substack{n+j < 0\\n-j < 0}}\Big )
(-1)^jq^{(K+1)n^2/2+(d+e-1)n/2-Kj^2/2+(d-e)j/2}\\
&= \sum_{n\ge0}q^{(K+1)n^2/2+(d+e-1)n/2}\sum_{j=-n}^{n}(-1)^jq^{-Kj^2/2+(d-e)j/2}\\
& \ \ \ \ \  -\sum_{n<0}q^{(K+1)n^2/2+(d+e-1)n/2}\sum_{n<j<-n}(-1)^jq^{-Kj^2/2+(d-e)j/2}\\
&= \sum_{n\ge0}q^{(K+1)n^2/2+(d+e-1)n/2}\sum_{j=-n}^{n}(-1)^jq^{-Kj^2/2+(d-e)j/2}\\
&\ \ \ \ \ \ -\sum_{n\ge0}q^{(K+1)(-n-1)^2/2+(d+e-1)(-n-1)/2}\sum_{-n-1<j<n+1}(-1)^jq^{-Kj^2/2+(d-e)j/2}\\
&= \sum_{n\ge0}q^{(K+1)n^2/2+(d+e-1)n/2}\sum_{j=-n}^{n}(-1)^jq^{-Kj^2/2+(d-e)j/2}\\
&\ \ \ \ \ \ -\sum_{n\ge0}q^{(K+1)n^2/2+(2K+3-d-e)n/2+(K+2-d-e)/2}\sum_{j=-n}^n(-1)^jq^{-Kj^2/2+(d-e)j/2}\\
&= \sum_{\sg(j)=\sg(n-j)}\sg(j)(-1)^jq^{(K+1)n^2/2+(d+e-1)n/2-Kj^2/2+(d-e)j/2}\\
& \ \ \ \ \ \cdot (1-q^{(K+2-d-e)n+(K+2-d-e)/2})\\
&=f_{K+1,K+1,1}(-q^{(K+d+e)/2},q^{d},q)\\
& \ \ \ \ \ -q^{(K+2-d-e)/2}f_{K+1,K+1,1}(-q^{2+(3K-d-e)/2},q^{K+2-e},q).
\end{align*}}%
For the second identity one replaces $(-1)^{(u-v)/2}$ with $(-1)^{(u+v)/2}$.  This also results in the $(-1)^j$ becoming $(-1)^n$. 
\end{proof}

\begin{proof} [Proof of Theorem \ref{theorem:main-result}]  We prove the identity for the initial plus sign.  The proof for the minus sign is analogous.  For (\ref{equation:OP-split-1}), we have
{\allowdisplaybreaks \begin{align*}
f_{1,2K+1,1}&(q^{1+\frac{1}{2}(m+\ell)},q^{1-\frac{1}{2}(m-\ell)},q)
+q^{-(\frac{K-m}{2})}f_{1,2K+1,1}(q^{1+\frac{1}{2}(2K-m+\ell)},q^{1-\frac{1}{2}(2K-m-\ell)},q)\\
&=f_{1,2K+1,1}(q^{1+\frac{1}{2}(m+\ell)},q^{1-\frac{1}{2}(m-\ell)},q)\\
& \ \ \ \ \  -q^{1+\frac{1}{2}(K+\ell)}f_{1,2K+1,1}(q^{1+K+1-\frac{1}{2}(m-\ell)},q^{1+K+1+\frac{1}{2}(m+\ell)},q)\\
 &=f_{1,2K+1,1}(q^{1+\frac{1}{2}(m+\ell)},q^{1-\frac{1}{2}(m-\ell)},q)\\
& \ \ \ \ \  -q^{1+\frac{1}{2}(K+\ell)}f_{1,2K+1,1}(q^{1+K+1+\frac{1}{2}(m+\ell)},q^{1+K+1-\frac{1}{2}(m-\ell)},q)\\
 &=f_{K+1,K+1,1}(q^{1+\frac{1}{2}(K+\ell)},q^{1+\frac{1}{2}(m+\ell)},q)\\
& \ \ \ \ \  +q^{\frac{1}{2}(K-\ell)}f_{K+1,K+1,1}(q^{1+K+\frac{1}{2}(K-\ell)},q^{1+K+\frac{1}{2}(m-\ell)},q),
\end{align*}}%
where the first equality follows (\ref{equation:f-shift}) with $(R,S)=(1,0)$ and the third from (\ref{equation:KP-disguise}).  \qedhere
\end{proof}

\begin{proof}[Proof of Corollaries \ref{corollary:f1K1-fKK1-A} and \ref{corollary:f1K1-fKK1-B}]  To prove (\ref{equation:OP-2}), we let $\ell=K$ in (\ref{equation:OP-split-1}).   To prove (\ref{equation:OP-3}), we set $m=K$ in (\ref{equation:OP-split-1}) to have
{\allowdisplaybreaks \begin{align*}
2C_{K,\ell}^{2K}(q)
&=\frac{q^{\lambda(K,\ell,2K)}}{J_{1}^3}\Big ( 
f_{K+1,K+1,1}(q^{1+\frac{1}{2}(K+\ell)},q^{1+\frac{1}{2}(K+\ell)},q)\\
& \ \ \ \ \ +q^{\frac{1}{2}(K-\ell)}f_{K+1,K+1,1}(q^{1+K+\frac{1}{2}(K-\ell)},q^{1+K+\frac{1}{2}(K-\ell)},q)\Big )\\
&=\frac{q^{\lambda(K,\ell,2K)}}{J_{1}^3}\Big ( 
f_{K+1,K+1,1}(q^{1+\frac{1}{2}(K+\ell)},q^{1-\frac{1}{2}(K-\ell)},q)\\
& \ \ \ \ \ +q^{\frac{1}{2}(K-\ell)}f_{K+1,K+1,1}(q^{1+K+\frac{1}{2}(K-\ell)},q^{1+K+\frac{1}{2}(K-\ell)},q)\Big )\\
&=\frac{q^{\lambda(K,\ell,2K)}}{J_{1}^3}\Big ( 
f_{K+1,K+1,1}(q^{1+\frac{1}{2}(K+\ell)},q^{1-\frac{1}{2}(K-\ell)},q)\\
& \ \ \ \ \ -q^{1-\frac{1}{2}(K-\ell)}f_{K+1,K+1,1}(q^{2+K+\frac{1}{2}(K+\ell)},q^{2-\frac{1}{2}(K-\ell)},q)\Big )\\
&=\frac{q^{\lambda(K,\ell,2K)}}{J_{1}^3}\Big ( 
2f_{K+1,K+1,1}(q^{1+\frac{1}{2}(K+\ell)},q^{1-\frac{1}{2}(K-\ell)},q)\Big ),
\end{align*}}%
where the second equality follows from (\ref{equation:f-shift}) with $(R,S)=(-1,1)$, the third equality follows from (\ref{equation:f-flip}), and the fourth follows from (\ref{equation:f-shift}) with $(R,S)=(0,-1)$.
\end{proof}


\section{Computing level $N=2$ string functions}  \label{section:LevelN2} 

\subsection{The string function  $c_{11}^{11}$ (\ref{equation:KP-2-orig})}\label{section:N2m1ell1}  It suffices to show (\ref{equation:KP-2-Hecke}):
\begin{equation}
f_{2,2,1}(q^2,q,q)=J_{1}J_{2}.
\end{equation}
In Corollary \ref{corollary:f221-expansion}, we find potential singularities, so we consider
\begin{equation*}
\lim_{x\rightarrow q}f_{2,2,1}(x^2,x,q).
\end{equation*}
We first focus on (\ref{equation:g221-id}) and write
{\allowdisplaybreaks \begin{align*}
\lim_{x\rightarrow q}h_{2,2,1}(x^2,x,q)
&=\lim_{x\rightarrow q}\Big [ j(x^2;q^2)m(-qx^{-1},q,-1)+j(x;q)m(q,q^2,-1)\Big ] \\
&=\lim_{x\rightarrow q}\Big [ j(x^2;q^2)m(-qx^{-1},q,-1)\Big ] \\
&=\lim_{x\rightarrow q} j(x^2;q^2)\Big [m(-qx^{-1},q,z)+\frac{zJ_{1}^3j(-z^{-1};q)j(qz/x;q)}
{j(-1;q)j(z;q)j(q/x;q)j(-qz/x;q)}\Big ]\\
&=\lim_{x\rightarrow q} j(x^2;q^2)\cdot \frac{zJ_{1}^3j(-z^{-1};q)j(qz/x;q)}
{j(-1;q)j(z;q)j(q/x;q)j(-qz/x;q)},
\end{align*}}%
where we have used (\ref{equation:changing-z-theorem}) in the penultimate equality.  Continuing with product rearrangements, we have
{\allowdisplaybreaks \begin{align*}
\lim_{x\rightarrow q}h_{2,2,1}(x^2,x,q)
&=\lim_{x\rightarrow q} \frac{j(x;q)j(-x;q)J_2}{J_1^2}\cdot \frac{zJ_{1}^3j(-z^{-1};q)j(qz/x;q)}
{j(-1;q)j(z;q)j(q/x;q)j(-qz/x;q)}\\
&=\lim_{x\rightarrow q} \frac{j(x;q)j(-x;q)J_2}{J_1^2}\cdot \frac{J_{1}^3j(-z;q)j(qz/x;q)}
{j(-1;q)j(z;q)j(x;q)j(-qz/x;q)}\\
&= J_{1}J_2,
\end{align*}}%
where in the second line we used (\ref{equation:j-elliptic}) and (\ref{equation:1.7}).  Now we consider 
{\allowdisplaybreaks \begin{align*}
&\lim_{x\rightarrow q}
\sum_{d=0}^{1}
\frac{q^{\binom{d+1}{2}}j\big (q^{1+d}x;q^{2}\big )  j\big (-q^{1-d}x;q^{2}\big ) 
J_{2}^3j\big (-q^{2+d}/x;q^{2}\big )}
{4\overline{J}_{1,4}\overline{J}_{2,8}j\big (-q;q^2\big )j\big (q^{1+d}/x;q^2\big )}\\
&=\lim_{x\rightarrow q}
\Big [ 
\frac{j\big (qx;q^{2}\big )  j\big (-qx;q^{2}\big ) 
J_{2}^3j\big (-q^{2}/x;q^{2}\big )}
{4\overline{J}_{1,4}\overline{J}_{2,8}\overline{J}_{1,2}j\big (q/x;q^2\big )}
+
\frac{q j\big (q^{2}x;q^{2}\big )  j\big (-x;q^{2}\big ) 
J_{2}^3j\big (-q^{3}/x;q^{2}\big )}
{4\overline{J}_{1,4}\overline{J}_{2,8}\overline{J}_{1,2}j\big (q^{2}/x;q^2\big )}
\Big ] \\
&=\lim_{x\rightarrow q}
\Big [ 
\frac{j\big (qx;q^{2}\big )  j\big (-qx;q^{2}\big ) 
J_{2}^3j\big (-x;q^{2}\big )}
{4\overline{J}_{1,4}\overline{J}_{2,8}\overline{J}_{1,2}j\big (qx;q^2\big )}
-
\frac{ j\big (x;q^{2}\big )  j\big (-x;q^{2}\big ) 
J_{2}^3j\big (-qx;q^{2}\big )}
{4\overline{J}_{1,4}\overline{J}_{2,8}\overline{J}_{1,2}j\big (x;q^2\big )}
\Big ] \\
&=\lim_{x\rightarrow q}
\Big [ 
\frac{  j\big (-qx;q^{2}\big ) 
J_{2}^3j\big (-x;q^{2}\big )}
{4\overline{J}_{1,4}\overline{J}_{2,8}\overline{J}_{1,2}}
-
\frac{  j\big (-x;q^{2}\big ) 
J_{2}^3j\big (-qx;q^{2}\big )}
{4\overline{J}_{1,4}\overline{J}_{2,8}\overline{J}_{1,2}}
\Big ]\\
&=0,
\end{align*}}%
where in the second equality we used (\ref{equation:j-elliptic}) and (\ref{equation:1.7}), and in the third equality we simplified.  Assembling the pieces, we have
\begin{equation}
\lim_{x\rightarrow q}f_{2,2,1}(x^2,x,q)=J_{1}J_{2}.
\end{equation}

\section{Computing level $N=4$ string functions}  \label{section:LevelN4}

\subsection{The string function  $c_{40}^{40}-c_{04}^{40}$ (\ref{equation:KP-4B-orig})}  It suffices to show (\ref{equation:KP-4B-Hecke}):
\begin{equation}
f_{3,3,1}(-q^2,q,q)-qf_{3,3,1}(-q^4,q^3,q)=J_1J_{1,2},
\end{equation}
which is true by Lemma \ref{lemma:f441-f331-evals}.


\section{Computing level $N=6$ string functions} \label{section:LevelN6} 

\subsection{The string function  $c_{51}^{33}$ (\ref{equation:KP-6A-orig})}  It suffices to show (\ref{equation:KP-6A-Hecke}):
\begin{equation}
f_{4,4,1}(q^4,q^3,q)=J_{2}J_{3,12}.
\end{equation}
In Corollary \ref{corollary:f441-expansion} we immediately see that
\begin{equation}
h_{4,4,1}(q^4,q^3,q)=0.
\end{equation}
Hence, it follows from Corollary \ref{corollary:f441-expansion} that
{\allowdisplaybreaks \begin{align*}
f_{4,4,1}&(q^4,q^3,q)\\
&= -\sum_{d=0}^3\frac{q^{3\binom{d+1}{2}}j(q^{6+3d};q^4)j(-q^{10-3d};q^{12})J_{12}^3j(-q^{3d};q^{12})}
{\overline{J}_{0,3}\overline{J}_{0,12}j(-q^{-2};q^{12})j(q^{2+3d};q^{12})}\\
&=
-q^{2}\frac{j(q^{6};q^4)\overline{J}_{2,12}J_{12}^3}
{\overline{J}_{0,3}\overline{J}_{2,12}J_{2,12}}
-\frac{q^{5}j(q^{9};q^4)\overline{J}_{7,12}J_{12}^3\overline{J}_{3,12}}
{\overline{J}_{0,3}\overline{J}_{0,12}\overline{J}_{2,12}J_{7,12}}
-\frac{q^{20}j(q^{15};q^4)J_{12}^3\overline{J}_{1,12}\overline{J}_{3,12}}
{\overline{J}_{0,3}\overline{J}_{0,12}\overline{J}_{2,12}J_{1,12}}\\
&=
q^{-1}\frac{J_{1,4}J_{12}^3\overline{J}_{3,12}}{\overline{J}_{0,3}\overline{J}_{0,12}\overline{J}_{2,12}}
\cdot \Big( \frac{\overline{J}_{1,12}}{J_{1,12}}-\frac{\overline{J}_{7,12}}{J_{7,12}}\Big ) 
+\frac{J_{2,4}J_{12}^3}
{\overline{J}_{0,3}J_{2,12}}\\
&=
q^{-1}\frac{J_{1,4}J_{12}^3\overline{J}_{3,12}}{\overline{J}_{0,3}\overline{J}_{0,12}\overline{J}_{2,12}}
\cdot \Big( \frac{\overline{J}_{1,12}J_{7,12}-J_{1,12}\overline{J}_{7,12}}{J_{1,12}J_{7,12}}\Big ) 
+\frac{J_{2,4}J_{12}^3}
{\overline{J}_{0,3}J_{2,12}}\\
&=
q^{-1}\frac{J_{1,4}J_{12}^3\overline{J}_{3,12}}{\overline{J}_{0,3}\overline{J}_{0,12}\overline{J}_{2,12}}
\cdot \Big( \frac{2qJ_{6,24}J_{4,24}}{J_{1,12}J_{7,12}}\Big ) 
+\frac{J_{2,4}J_{12}^3}
{\overline{J}_{0,3}J_{2,12}}\\
&=J_{2}J_{3,12},
\end{align*}}%
where we have used (\ref{equation:j-elliptic}), collected terms, and then used (\ref{equation:H1Thm1.2A}).

\subsection{The string function $c_{51}^{51}+c_{15}^{51}$ (\ref{equation:KP-6B-orig})}  It suffices to show (\ref{equation:KP-6B-Hecke}):
\begin{equation}
f_{4,4,1}(q^3,q^2,q)+qf_{4,4,1}(q^5,q^4,q)=J_{2}J_{6,12}.
\end{equation}
We use Corollary \ref{corollary:f441-expansion}.  Considering (\ref{equation:h441}), we see that two of the four summands vanish right away giving us
{\allowdisplaybreaks \begin{align*}
h_{4,4,1}(q^3,q^2,q)+qh_{4,4,1}(q^5,q^4,q)
&=j(q^3;q^4)m(-q^4,q^5,-1)+qj(q^5;q^4)m(-q^4,q^5,-1)\\
&=j(q^3;q^4)m(-q^4,q^5,-1)-j(q;q^4)m(-q^4,q^5,-1)\\
&=0,
\end{align*}}%
where we have used (\ref{equation:j-elliptic}) and (\ref{equation:1.7}).   Thus we have
{\allowdisplaybreaks \begin{align*}
f_{4,4,1}&(q^3,q^2,q)+qf_{4,4,1}(q^5,q^4,q)\\
&=-\sum_{d=0}^3\frac{q^{3\binom{d+1}{2}}j(q^{5+3d};q^4)j(-q^{10-3d};q^{12})J_{12}^3j(-q^{3+3d};q^{12})}{\overline{J}_{0,3}\overline{J}_{0,12}j(-q;q^{12})j(q^{2+3d};q^{12})}\\
& \ \ \ \ \ -q\sum_{d=0}^3\frac{q^{3\binom{d+1}{2}}j(q^{7+3d};q^4)j(-q^{10-3d};q^{12})J_{12}^3j(-q^{-3+3d};q^{12})}{\overline{J}_{0,3}\overline{J}_{0,12}j(-q^{-5};q^{12})j(q^{2+3d};q^{12})}\\
&= \frac{J_{12}^3}{\overline{J}_{0,3}\overline{J}_{0,12}}\Big [ q^{-1}\frac{J_{1,4}\overline{J}_{2,12}\overline{J}_{3,12}}{\overline{J}_{1,12}J_{2,12}}+0
-q^{-1}\frac{J_{1,4}\overline{J}_{4,12}\overline{J}_{3,12}}{\overline{J}_{1,12}J_{4,12}}
+\frac{J_{2,4}\overline{J}_{0,12}}{J_{1,12}}\Big ] \\
& \ \ \ \ \ +\frac{J_{12}^3}{\overline{J}_{0,3}\overline{J}_{0,12}}
\Big [ \frac{J_{1,4}\overline{J}_{2,12}\overline{J}_{3,12}}{\overline{J}_{5,12}J_{2,12}}
-q\frac{J_{2,4}\overline{J}_{0,12}}{J_{5,12}}
+\frac{J_{1,4}\overline{J}_{4,12}\overline{J}_{3,12}}{\overline{J}_{5,12}J_{4,12}}+0\Big ] ,
\end{align*}}%
where we have used (\ref{equation:j-elliptic}), (\ref{equation:1.7}), and simplified.  Regrouping terms, using (\ref{equation:1.7}), and simplifying, we have
{\allowdisplaybreaks \begin{align*}
f_{4,4,1}&(q^3,q^2,q)+qf_{4,4,1}(q^5,q^4,q)\\
&= q^{-1}\frac{J_{12}^3J_{1,4}\overline{J}_{3,12}}{\overline{J}_{0,3}\overline{J}_{0,12}\overline{J}_{1,12}}
\cdot \frac{\overline{J}_{2,12}J_{4,12}-\overline{J}_{4,12}J_{2,12}}{J_{2,12}J_{8,12}}
+\frac{J_{12}^3J_{2,4}}{\overline{J}_{0,3}}\cdot \frac{J_{5,12}-qJ_{11,12}}{J_{1,12}J_{7,12}} \\
& \ \ \ \ \ +\frac{J_{12}^3J_{1,4}\overline{J}_{3,12}}{\overline{J}_{0,3}\overline{J}_{0,12}\overline{J}_{5,12}}
\cdot \frac{\overline{J}_{2,12}J_{4,12}+\overline{J}_{4,12}J_{2,12}}{J_{2,12}J_{8,12}}\\
&= q^{-1}\frac{J_{12}^3J_{1,4}\overline{J}_{3,12}}{\overline{J}_{0,3}\overline{J}_{0,12}\overline{J}_{1,12}}
\cdot \frac{J_6}{J_{2,6}J_{12}^2}\cdot \Big ( \overline{J}_{2,12}J_{4,12}-\overline{J}_{4,12}J_{2,12}\Big ) \\
& \ \ \ \ \ +\frac{J_{12}^3J_{1,4}\overline{J}_{3,12}}{\overline{J}_{0,3}\overline{J}_{0,12}\overline{J}_{5,12}}
\cdot \frac{J_6}{J_{2,6}J_{12}^2}
\cdot \Big ( \overline{J}_{2,12}J_{4,12}+\overline{J}_{4,12}J_{2,12}\Big ) \\
& \ \ \ \ \ +\frac{J_{12}^3J_{2,4}}{\overline{J}_{0,3}}\cdot \frac{J_6}{J_{1,6}J_{12}^2}\cdot 
\Big ( J_{5,12}-qJ_{11,12}\Big ),
\end{align*}}%
where we have used (\ref{equation:1.10}).  Employing (\ref{equation:H1Thm1.2A}), (\ref{equation:H1Thm1.2B}), and (\ref{equation:j-split}), yields
{\allowdisplaybreaks \begin{align*}
f_{4,4,1}&(q^3,q^2,q)+qf_{4,4,1}(q^5,q^4,q)\\
&= q^{-1}\frac{J_{12}^3J_{1,4}\overline{J}_{3,12}}{\overline{J}_{0,3}\overline{J}_{0,12}\overline{J}_{1,12}}
\cdot \frac{J_6}{J_{2,6}J_{12}^2}\cdot 2q^2J_{2,24}J_{18,24} 
+\frac{J_{12}^3J_{1,4}\overline{J}_{3,12}}{\overline{J}_{0,3}\overline{J}_{0,12}\overline{J}_{5,12}}
\cdot \frac{J_6}{J_{2,6}J_{12}^2}
\cdot 2J_{6,24}J_{14,24} \\
& \ \ \ \ \ +\frac{J_{12}^3J_{2,4}}{\overline{J}_{0,3}}\cdot \frac{J_6}{J_{1,6}J_{12}^2}\cdot 
j(q;-q^3)\\
&= 2\frac{J_{12}^3J_{1,4}\overline{J}_{3,12}}{\overline{J}_{0,3}\overline{J}_{0,12}}
\cdot \frac{J_6J_{6,24}}{J_{2}J_{12}^2}\cdot \frac{J_{24}}{J_{12}^2}\cdot \Big ( qJ_{1,12}+J_{5,12}\Big )
 +\frac{J_{12}^3J_{2,4}}{\overline{J}_{0,3}}\cdot \frac{J_6}{J_{1,6}J_{12}^2}\cdot 
j(q;-q^3)\\
&= 2\frac{J_{12}^3J_{1,4}\overline{J}_{3,12}}{\overline{J}_{0,3}\overline{J}_{0,12}}
\cdot \frac{J_6J_{6,24}}{J_{2}J_{12}^2}\cdot \frac{J_{24}}{J_{12}^2}\cdot j(-q;-q^3)
 +\frac{J_{12}^3J_{2,4}}{\overline{J}_{0,3}}\cdot \frac{J_6}{J_{1,6}J_{12}^2}\cdot 
j(q;-q^3),
\end{align*}}%
where for the second equality we used (\ref{equation:1.12}), and for the third equality we used (\ref{equation:j-split}).  Applying (\ref{equation:1.11}) gives
{\allowdisplaybreaks \begin{align*}
f_{4,4,1}&(q^3,q^2,q)+qf_{4,4,1}(q^5,q^4,q)\\
&= 2\frac{J_{12}^3J_{1,4}\overline{J}_{3,12}}{\overline{J}_{0,3}\overline{J}_{0,12}}
\cdot \frac{J_6J_{6,24}}{J_{2}J_{12}^2}\cdot \frac{J_{24}}{J_{12}^2}\cdot \frac{\overline{J}_{1,6}J_{2,6}}{J_{3,12}}
 +\frac{J_{12}^3J_{2,4}}{\overline{J}_{0,3}}\cdot \frac{J_6}{J_{1,6}J_{12}^2}\cdot 
\frac{J_{1,6}\overline{J}_{2,6}}{J_{3,12}}\\
&=\frac{1}{2}J_{2}J_{6,12}+\frac{1}{2}J_{2}J_{6,12}\\
&=J_2J_{6,12},
\end{align*}}%
where the second equality follows from elementary product rearrangements.

\subsection{The string function $c_{51}^{51}-c_{15}^{51}$ (\ref{equation:KP-6C-orig})}  It suffices to show (\ref{equation:KP-6C-Hecke}):
\begin{equation}
f_{4,4,1}(-q^3,q^2,q)-qf_{4,4,1}(-q^5,q^4,q)=J_1^2.
\end{equation}
From Lemma \ref{lemma:f441-f331-evals}, we have
\begin{align*}
J_1^2
&=f_{4,4,1}(-q^{3},q,q)-qf_{4,4,1}(-q^{5},q^{3},q)\\
&=-qf_{4,4,1}(-q^7,q^2,q)+j(-q^3;q^4)+q^4f_{4,4,1}(-q^{9},q^{4},q)-qj(-q^5;q^4)\\
&=-qf_{4,4,1}(-q^5,q^4,q)+f_{4,4,1}(-q^{3},q^{2},q),
\end{align*}
where for the second equality we used (\ref{equation:f-shift}) with $(R,S)=(0,1)$, and for the third we used (\ref{equation:j-elliptic}) and (\ref{equation:f-flip}).


\section{Computing level $N=8$ string functions} \label{section:LevelN8} 
\subsection{The string function  $c_{62}^{44}$ (\ref{equation:KP-8A-orig})}  It suffices to show  (\ref{equation:KP-8A-Hecke}):
\begin{equation}
f_{5,5,1}(q^5,q^4,q)=J_{2}J_{10},
\end{equation}
but this is just the first identity in Lemma \ref{lemma:f441-f331-evals}.

\section{Computing level $N=10$ string functions} \label{section:LevelN10} 

\subsection{The string function $c_{91}^{55}$ (\ref{equation:KP-10B-orig})}  It suffices to show (\ref{equation:KP-10B-Hecke}):
\begin{equation}
f_{6,6,1}(q^6,q^4,q)=J_{4,10}J_{3,15}.
\end{equation}
We use Corollary \ref{corollary:f661-expansion}.  In (\ref{equation:h661}), we have
\begin{equation}
h_{6,6,1}(q^6,q^4,-1,-1)
=j(q^6;q^6)m( -q^{3},q^{5},-1 )
+j(q^4;q)m( q^{-3},q^{30},-1  )=0.
\end{equation}
Hence
{\allowdisplaybreaks \begin{align*}
f_{6,6,1}&(q^6,q^4,q)\\
&=-\frac{J_{30}^3}{\overline{J}_{0,5}\overline{J}_{0,30}}\sum_{d=0}^{5}
q^{5\binom{d+1}{2}}
 \cdot \frac{j\big (q^{5d+9};q^{6}\big )  j\big (-q^{27-5d};q^{30}\big )j\big (-q^{5d};q^{30}\big )}
{j\big (-q^{-3};q^{30})j(q^{5d+3};q^{30}\big )}\\
&=\frac{J_{30}^3}{\overline{J}_{0,5}\overline{J}_{0,30}\overline{J}_{3,30}}\Big [ 
\frac{J_{3,6}\overline{J}_{27,30}\overline{J}_{0,30}}{J_{3,30}}
- q^{-2} \cdot \frac{J_{2,6} \overline{J}_{22,30}\overline{J}_{5,30}}{J_{8,30}}\\
&\ \ \ \ \ +q^{-3} \cdot \frac{J_{1,6} \overline{J}_{17,30}\overline{J}_{10,30}}{J_{13,30}}
-q^{-3} \cdot \frac{J_{1,6} \overline{J}_{7,30}\overline{J}_{20,30}}{J_{23,30}}
+q^{-2} \cdot \frac{J_{2,6} \overline{J}_{2,30}\overline{J}_{25,30}}{J_{28,30}}
\Big ],
\end{align*}}%
where we have used (\ref{equation:j-elliptic}).  Using (\ref{equation:1.7}) and regrouping terms yields
{\allowdisplaybreaks \begin{align*}
f_{6,6,1}&(q^6,q^4,q)\\
&=\frac{J_{30}^3}{\overline{J}_{0,5}\overline{J}_{0,30}\overline{J}_{3,30}}\Big [ 
\frac{J_{3,6}\overline{J}_{27,30}\overline{J}_{0,30}}{J_{3,30}}
+q^{-2} \cdot J_{2} \overline{J}_{5,30}\cdot 
\frac{\overline{J}_{2,30}J_{8,30}-\overline{J}_{8,30}J_{2,30}}{J_{2,30}J_{8,30}}\\
& \ \ \ \ \ +q^{-3} \cdot J_{1,6} \overline{J}_{10,30}\cdot 
\frac{\overline{J}_{13,30}J_{7,30}-\overline{J}_{7,30}J_{13,30}}{J_{7,30}J_{13,30}}\Big ] \\
&=\frac{J_{30}^3}{\overline{J}_{0,5}\overline{J}_{0,30}\overline{J}_{3,30}}\Big [ 
\frac{J_{3,6}\overline{J}_{27,30}\overline{J}_{0,30}}{J_{3,30}}
+2\cdot 
\frac{J_{2} \overline{J}_{5,30}J_{6,60}J_{40,60}}{J_{2,30}J_{8,30}}
 -2q^{4} \cdot  
\frac{J_{1,6} \overline{J}_{10,30}J_{6,60}J_{50,60}}{J_{7,30}J_{13,30}}\Big ],
\end{align*}}%
where we have used (\ref{equation:H1Thm1.2A}).  We focus on the last two summands:
{\allowdisplaybreaks \begin{align*} 
&\frac{J_{2,6} \overline{J}_{5,30}J_{6,60}J_{40,60}}{J_{2,30}J_{8,30}}
 -q^{4} \cdot \frac{J_{1,6} \overline{J}_{10,30}J_{6,60}J_{50,60}}{J_{7,30}J_{13,30}}\\
&=J_{6,60}\overline{J}_{5,30}\overline{J}_{10,30}\cdot \frac{J_{60}}{J_{30}^2}\cdot \Big [ \frac{J_{2,6} J_{10,30}}{J_{2,30}J_{8,30}}
-q^{4} \cdot \frac{J_{1,6}J_{5,30}}{J_{7,30}J_{13,30}}\Big ] \\
&=J_{6,60}\overline{J}_{5,30}\overline{J}_{10,30}\cdot \frac{J_{6}}{J_{30}^5}
\cdot\frac{J_{60}}{J_{30}^2}\cdot \Big [ J_{14,30}J_{20,30}J_{26,30} J_{10,30}
-q^{4} \cdot J_{1,30}J_{19,30}J_{25,30}J_{5,30}\Big ].
\end{align*}}%
If we specialize Proposition \ref{proposition:Weierstrass} with $a=q^{15}$, $b=q^{10}$, $c=q^{6}$, $d=q^{5}$, we have
\begin{equation}
J_{16,30}J_{20,30}J_{26,30} J_{10,30}
+q^{4} \cdot J_{1,30}J_{11,30}J_{25,30}J_{5,30}=J_{21,30}J_{9,30}J_{15,30}J_{5,30}.
\end{equation}
Using (\ref{equation:1.7}) we can then write 
\begin{align*} 
f_{6,6,1}(q^6,q^4,q)
&=\frac{J_{30}^3}{\overline{J}_{0,5}}  \frac{J_{3,6}}{J_{3,30}}
+2\cdot  \frac{J_{6,60}\overline{J}_{5,30}\overline{J}_{10,30}}{\overline{J}_{0,5}\overline{J}_{0,30}\overline{J}_{3,30}}
\cdot \frac{J_{6}J_{60}}{J_{30}^4} \cdot \\
& \ \ \ \ \ \cdot \Big [ 2J_{16,30}J_{20,30}J_{26,30} J_{10,30}
- J_{9,30}J_{21,30}J_{15,30}J_{5,30}\Big ].
\end{align*}
The result follows from the two identities in Lemma \ref{lemma:f661-identities}.

\section*{Conclusion}
We have derived the expressions for even-level string functions for an irreducible highest weight module of  Kac--Moody Lie algebra $A_1^{(1)}$ in terms of double sums of the form $f_{K+1,K+1,1}(x,y,q)$. These double-sums may be expressed  in terms of  theta functions and Appell--Lerch functions: the building blocks of Ramanujan's  mock theta functions.

\section*{Acknowledgements}

This research was supported by the Theoretical Physics and Mathematics Advancement Foundation BASIS, agreement No. 20-7-1-25-1.  We would also like to thank O. Warnaar and E. Feigin for helpful comments and suggestions.

\end{document}